\documentclass[reqno]{amsart}

\usepackage{amsmath, amssymb, amsthm, epsfig}
\usepackage[hidelinks]{hyperref}
\usepackage{latexsym}
\usepackage{url}
\usepackage[mathscr]{euscript}
\usepackage{color}
\usepackage{harpoon}
\usepackage{url}
\usepackage{enumitem} 
\let\Horig\H
\usepackage[T1]{fontenc}
\usepackage{fullpage} 
\usepackage{setspace}
\usepackage{refcheck} 
\norefnames
\nocitenames

\def\today{\ifcase\month\or
  January\or February\or March\or April\or May\or June\or
  July\or August\or September\or October\or November\or December\fi
  \space\number\day, \number\year}

\newtheorem{theorem}{Theorem}

\newtheorem{lemma}[theorem]{Lemma}
\newtheorem{corollary}[theorem]{Corollary}
\newtheorem{remark}{Remark}

\newcommand{\A}{\mathcal{A}}

\renewcommand{\H}{\mathcal{H}}

\renewcommand{\L}{\mathcal{L}}
\newcommand{\M}{\mathcal{M}}

\renewcommand{\P}{\mathcal{P}}
\newcommand{\Q}{\mathcal{Q}}

\newcommand{\Z}{\mathcal{Z}}

\newcommand{\n}{\mathbb{N}}

\renewcommand{\r}{\mathbb{R}}
\newcommand{\re}{{\rm Re}\,}
\newcommand{\ft}{\widehat}

\newcommand{\bo}{\boldsymbol}

\newcommand{\wt}{\widetilde}

\newcommand{\la}{\lambda}
\newcommand{\ga}{\gamma}

\newcommand{\ep}{\varepsilon}

\newcommand{\hh}{\tfrac12}
\renewcommand{\d}{\,\text{\rm d}}

\begin{document}

%------------------HEADINGS------------------------

\title[]{Pair Correlation Estimates for the Zeros of the Zeta Function via Semidefinite Programming}
\author[Chirre, Gon\c{c}alves and de Laat]{Andr\'es Chirre, Felipe Gon\c{c}alves and David de Laat}
\date{\today}
\subjclass[2010]{}
\address{IMPA - Instituto Nacional de Matem\'{a}tica Pura e Aplicada - Estrada Dona Castorina, 110, Rio de Janeiro, RJ, Brazil 22460-320}
\email{achirre@impa.br}
\address{Hausdorff Center for Mathematics, Universit\"at Bonn, Endenicher Allee 60, 53115 Bonn, Germany}
\email{goncalve@math.uni-bonn.de}
\address{Massachusetts Institute of Technology, Department of Mathematics, 77 Massachusetts Avenue, Simons building (Building 2), Room 2-241}
\email{mail@daviddelaat.nl}
%\numberwithin{theorem}{subsection}{section}
%\numberwithin{equation}{section}

%------------------ABSTRACT------------------------------

\begin{abstract}
In this paper we study the distribution of the non-trivial zeros of the Riemann zeta-function $\zeta(s)$ (and other L-functions) using Montgomery's pair correlation approach. We use semidefinite programming to improve upon numerous asymptotic
bounds in the theory of $\zeta(s)$, including the proportion of
distinct zeros, counts of small gaps between zeros, and sums involving multiplicities of zeros.
\end{abstract}

%---------------------TITLE--------------------------------

\maketitle

%---------------------HAVE--FUN!-----------------------------------

\section{Introduction}

In this paper we give improved asymptotic bounds for several quantities related to the zeros of the Riemann zeta-function (and other functions) using Montgomery's pair correlation approach \cite{M}. The key idea is to replace the usual bandlimited auxiliary functions by the class of functions used in the linear programming bounds developed by Cohn and Elkies \cite{CE} for the sphere packing problem. The advantage of this framework is that it reduces the problems to convex optimization problems that can be solved numerically via semidefinite programming. For all problems we considered this produces better bounds than any bandlimited construction.

\subsection{Background} Let $\zeta(s)$ be the Riemann zeta-function. It is well known that all non-trivial zeros of $\zeta(s)$ are located in the critical strip $0<\re s<1$, and the Riemann Hypothesis (RH) is the statement that all these zeros are aligned on the line $\re s=1/2$. Let $N(T)$ count the number of zeros $\rho = \beta + i \gamma$ of $\zeta(s)$, repeated according the multiplicity, such that $0<\beta<1$ and $0 <\gamma \leq T$. The Riemann-von Mangoldt formula (in its weaker form) states that
\begin{align}\label{RvMformula}
N(T) =  (1+o(1)) \frac{T}{2\pi} \log T.
\end{align}
Let
\[
N^{*}(T):=\displaystyle\sum_{0<\gamma\leq T}m_{\rho},
\]
where the sum is over the non-trivial zeros of $\zeta(s)$ counting multiplicities\footnote{For every sum over zeros in this article the involved quantities should be repeated according to the multiplicity of the zero.} and $m_\rho$ is the multiplicity of $\rho$. In addition to RH, it is also conjectured that all zeros of $\zeta(s)$ are simple, therefore it is {\it conjectured} that
\begin{align} \label{NstarNconjecure}
N^{*}(T) \sim N(T),
\end{align}
as $T\to\infty$. To study the distribution of the zeros of the Riemann zeta-function, Montgomery defined the pair correlation function
\begin{equation}\label{paircorrelfunction}
N(x,T):=\sum_{\substack{0< \ga,\ga'\leq T \\ 0<\ga'-\ga\leq \frac{2\pi x}{\log T}}} 1
\end{equation}
for $x > 0$ and {\it conjectured} that
\[N(x,T) \sim N(T)\int_0^x \bigg(1-\frac{\sin^2(\pi y)}{(\pi y)^2}\bigg)\d y,
\]
as $T \to \infty$. Note that by \eqref{RvMformula} the average gap between zeros is $\frac{2\pi}{\log T}$, hence $N(x,T)$ is counting pairs of zeros not greater than $x$ times the average gap.

One line of research to understand and give evidence for the conjectures above is to produce bounds of the form
\begin{equation}\label{prob1}
N^*(T) \leq (c + o(1)) \, N(T),
\end{equation}
and
\begin{equation}\label{prob2}
N(x,T) \gg N(T),
\end{equation}
with $c > 1$ and $x>0$ as small as possible, as $T\to\infty$. 

These two problems have been widely studied with several improvements being made over the years. One approach is to use a formula relating sums and integrals involving an auxiliary function $f$ from a class $\A$, and then use this to derive an inequality involving the quantities we want to compare and the value of some functional $Q$ evaluated at $f$. Minimizing (or maximizing) the functional over the class $\A$ would then produce the best possible bound with that specific approach. Nowadays, this idea is a standard technique in analytic number theory (introduced first by Beurling and Selberg) and the following are some references (clearly not a complete list) where the main approach is exactly that: Large sieve inequalities \cite{GV, HV}; Erd\Horig{o}s-Tur\'{a}n inequalities \cite{CV2, Va}; Hilbert-type inequalities \cite{CL3, CLV, CV2, GMK, GV, Va}; Tauberian theorems \cite{GV}; Bounds in the theory of the Riemann zeta-function and L-functions \cite{CC,CChi,CChiM,CCLM,CCM,CCM2,CF,CS,Chi,Ga,GG}; Prime gaps \cite{CMS}.

For problem \eqref{prob1} Montgomery \cite{M} uses Fourier analysis to derive the inequality
\[
N^*(T) \leq \frac{1}{f(0)}\left( \widehat f(0) + \int_{-1}^1 \widehat f(x) |x| \d x + o(1) \right) N(T),
\]
for any non-negative function $f \in L^1(\r)$ with $\ft f$ supported in $[-1,1]$, where
\[
\widehat f(x) = \int_{-\infty}^{\infty} f(y) e^{-2\pi xy} \d y.
\]
Montgomery then gives a function $f$ that proves the bound \eqref{prob1} with $c \leq 4/3$. In \cite{MT} the optimal function in this class is found and as mentioned in \cite{CG} gives the bound $c \le 1.3275$. We relax the condition on the support of $\ft f$ to the requirement $\ft f(x) \le 0$ for $|x| \ge 1$, which matches exactly with the conditions required by the linear programming bounds for the sphere packing problem (see Section \ref{prelims} for a detailed explanation). This connection is what ultimately inspired us to attack the problem numerically and to find good test functions for the functionals derived in Section \ref{prelims}. From our point of view, our main contribution is the realization that methods from the sphere packing problem can applied in the theory of the Riemann zeta-function to improve several asymptotic bounds and, to the best of our knowledge, it is the first time it has been done.
\newpage

\section{Main Results}
We now state our main results.
\begin{theorem} \label{thmNstar}
	Assuming RH, we have
	\begin{align*}
	N^*(T)\leq ({1.3208}+o(1))N(T).
	\end{align*}
	Assuming the Generalized Riemann Hypothesis for Dirichlet L-functions (GRH), we have
	\begin{align*}
	N^*(T)\leq ({1.3155}+o(1))N(T).
	\end{align*}
\end{theorem}

Montgomery \cite{M} was the first to show the constant $4/3 = 1.3333...$. Later Montgomery and Taylor \cite{MT} improved on this and found the bound  $1.3275$ as mentioned by Cheer and Goldston in \cite{CG}. Assuming the generalized Riemann Hypothesis GRH, Goldston, Gonek, \"{O}zl\"{u}k and Snyder \cite{GGOS} improved it to $1.3262$. 

Theorem~\ref{thmNstar} has an important application to estimating the quantity of simple zeros of $\zeta(s)$. Let
\[
N_s(T):=\displaystyle\sum_{\substack{0<\gamma\leq T\\ m_\rho=1}}1.
\]
Using the fact that
\begin{equation}\label{eq:firstineq}
{N_s(T)\geq \displaystyle\sum_{0<\ga\leq T}(2-m_\rho)}=2N(T)-N^{*}(T).
\end{equation}
we obtain the following corollary.

\begin{corollary} 
Assuming RH, we have
\[
{N_s(T)}\geq ({0.6792} + o(1)){N(T)}.%{2-??}
\]
Assuming GRH, we have
\[
{N_s(T)}\geq ({0.6845} + o(1)){N(T)}. %{2-??}
\]
\end{corollary}

Using the pair correlation approach, the previous best result known is due by Cheer and Goldston \cite{CG} showing that $67.27\%$ of the zeros are simple. Assuming GRH, Goldston, Gonek, \"{O}zl\"{u}k and Snyder \cite{GGOS} showed that $67.38\%$ are simple. In this way, we improved all these bounds. However, by a different technique, still assuming RH, Bui and Heath-Brown \cite{BHB} improved the result to $19/27=70.37...\%$, which currently is the best.

Combining the above result of Bui and Heath-Brown with Theorem \ref{thmNstar} and an argument of Ghosh, we can bound the proportion of distinct zeros of the Riemann zeta-function. Let
\[
N_d(T) := \sum_{0<\gamma\leq T} \frac{1}{m_\rho},
\]
be the number of distinct zeros of $\zeta(s)$ with $0<\ga \leq T$. Using the inequality
\begin{equation}\label{eq:secondineq}
2N_s(T)\leq \displaystyle\sum_{0<\gamma\leq T}\dfrac{(m_\rho-2)(m_\rho-3)}{m_\rho}=N^{*}(T)-5N(T)+6N_d(T).
\end{equation}
in conjunction with the estimate
\[
N_s(T) \geq \left(\frac{19}{27}+o(1)\right)N(T)
\]
and Theorem \ref{thmNstar}, we deduce the following corollary.

\begin{corollary} 
Assuming RH, we have
\[
{N_d(T)}\geq ({0.8477} + o(1)){N(T)}.%{(2*19/27 + 5 - ??)/6}
\]
Assuming GRH, we have
\[
{N_d(T)}\geq ({0.8486} + o(1)){N(T)}. %\remark{(2*19/27+ 5 - ??)/6}
\]
\end{corollary}

Using the pair correlation approach, the best previous result known is due to Farmer, Gonek and Lee \cite{FGL} with constant $0.8051$. By a different technique, assuming RH, Bui and Heath-Brown \cite{BHB} improved the constant to $0.8466$.

We also obtain the best known results for the minimal nonzero value of Montgomery's pair correlation function.

\begin{theorem} \label{thmpaircorreltation}
Assuming RH and \eqref{NstarNconjecure}, we have
\[
N({0.6039},T)\gg N(T).
\]
Assuming GRH and \eqref{NstarNconjecure}, we have
\[
N(0.5769,T)\gg N(T).
\]
\end{theorem}

Montgomery \cite{M} showed that $N(0.68..., T)\gg N(T)$, and in \cite{GGOS} it is pointed out that it is not difficult to modify Montgomery's argument to derive the sharper constant $0.6695$. This result was improved by Goldston, Gonek, \"{O}zl\"{u}k and Snyder \cite{GGOS} with constant $0.6072$. Later,  Carneiro, Chandee, Littmann and Milinovich \cite{CCLM} improved the constant to $0.6068...$. Assuming GRH and \eqref{NstarNconjecure}, Goldston, Gonek, \"{O}zl\"{u}k and Snyder showed the constant $0.5781...$. 

\subsection{Results for zeros of Dirichlet L-functions} To obtain averaged bounds for the percentage of simple zeros of primitive Dirichlet L-functions we use the framework established by Chandee,  Lee, Liu and Radziwi\l\l \ \cite{CLLR}. Let $\Phi$ be a real-valued smooth function supported in the interval $[a,b]$ with $0<a<b<\infty$. Define its Mellin transform by
\[
\mathcal{M}\Phi(s)=\int_{0}^{\infty}\Phi(x)x^{s-1}\d x.
\]
For a character $\chi$ mod $q$, let $L(s,\chi)$ be its associated Dirichlet L-function. Under GRH, all non-trivial zeros of $L(s,\chi)$ lie on the critical line $\re s = 1/2$.
Let 
$$
N_{\Phi}(Q):=\displaystyle\sum_{Q\leq q\leq 2Q}\dfrac{W(q/Q)}{\varphi(q)}\displaystyle\sum_{\substack{\chi  \ (\mathrm{mod} \ q) \\ \text{primitive}}} \displaystyle\;\sum_{\gamma_\chi}\big|\mathcal{M}\Phi(i\gamma_\chi)\big|^2,
$$
where $W$ is a non-negative smooth function supported in $(1,2)$, and where the last sum is over all non-trivial zeros $\hh+i\gamma_\chi$ of the Dirichlet L-function $L(s,\chi)$. In \cite[Lemma 2.1]{CLLR} it is shown that
$$ 
N_{\Phi}(Q)\sim\dfrac{A}{2\pi}Q\log Q \int_{-\infty}^{\infty}\big|\mathcal{M}\Phi(ix)\big|^2\d x,
$$ 
where 
$$
A=\M W(1)\displaystyle\prod_{p\ \text{prime}}\bigg(1-\dfrac{1}{p^2}-\dfrac{1}{p^3}\bigg).
$$
Let
$$
N_{\Phi,s}(Q)=\displaystyle\sum_{Q\leq q\leq 2Q}\dfrac{W(q/Q)}{\varphi(q)}\displaystyle\sum_{\substack{\chi  \ (\mathrm{mod} \ q) \\ \text{primitive}}} \displaystyle\;\sum_{\substack{\gamma_\chi \\ \text{simple}}}\big|\mathcal{M}\Phi(i\gamma_\chi)\big|^2.
$$
The quantity 
$$
\liminf_{Q\to\infty}\frac{N_{\Phi,s}(Q)}{N_{\Phi}(Q)}
$$
then measures (in average) the proportion of simple zeros among all primitive Dirichlet L-functions. 

In addition, for the following theorem, we require that $\Phi(x)$ and $\M \Phi(ix)$ are non-negative functions. We note that we can also further relax the conditions on $\Phi$ so to include the function given by $\M\Phi(ix) = (\sin x/x)^2$, as was established in \cite{CLLR} and \cite{S}. 

\begin{theorem} \label{thmsimplezerosofdir} 
Assuming GRH, we have
\[
N_{\Phi,s}(Q)\geq (0.9350+o(1))N_{\Phi}(Q). 
\]
\end{theorem}

Using the pair correlation approach, Chandee, Lee, Liu and Radziwi\l\l \ \cite{CLLR} showed that $91.66\%$ of the zeros are simple. The best previous result known is due to Sono \cite{S}, showing that $93.22\%$ of the zeros are simple. On another hand, \"{O}zl\"{u}k \cite{O} obtained a similar lower bound but for all Dirichlet L-functions rather than just the primitive L-functions, showing that $91.66\%$ of the zeros are simple (in some sense).

\subsection{Results for zeros of $\xi'(s)$}
We can extend our analysis to the zeros of $\xi'(s)$, where
\begin{align*}
\xi(s)=\frac{1}{2}s(s-1)\pi^{-\frac{s}{2}}\Gamma\bigg(\frac{s}{2}\bigg)\zeta(s).
\end{align*}
It is known that $\xi'(s)$ has only zeros in the critical strip $0<\re s<1$ and that RH implies that all its zeros are also on the line $\re s=1/2$. Let $N_1(T)$ count the number of zeros $\rho_1=\beta_1+i\gamma_1$ of $\xi'(s)$ (with multiplicity) such that $0< \gamma_1 \leq T$. It is also known that
\begin{align*}
N_1(T) =  (1+o(1)) \frac{T}{2\pi} \log T.
\end{align*}
We can then similarly define the function 
$$
N_{1}^{*}(T):=\displaystyle\sum_{0<\gamma_1\leq T}m_{\rho_1},
$$
where $m_{\rho_1}$ is the multiplicity of the zero $\rho_1$, and derive the sharpest known upper bound for $N_1^*(T)$.
\begin{theorem}\label{thmn1star}
	Assuming RH, we have
	\begin{align*}
	N_1^*(T)\leq (1.1175+o(1))N_1(T).
	\end{align*}
\end{theorem}

Defining the functions $N_{1,s}(T)$ and $N_{1,d}(T)$ (quantity of simple and distinct zeros respectively) for $\xi'(s)$ and using the inequalities
\[
{N_{1,s}(T) \geq 2 N_1(T) - N_1^*(T)}
\]
and
\[
{N_{1,d}(T) \geq \frac{3}{2}N_1(T)-\frac{1}{2}N_1^*(T)},
\]
that can be derived using the analogues of \eqref{eq:firstineq} and \eqref{eq:secondineq} for $\xi'(s)$, we obtain the following corollary. 

\begin{corollary} 
Assuming RH, we have
\[	
N_{1,s}(T)\geq (0.8825+o(1))N_1(T). % 2-??
\]
and
\[
N_{1,d}(T)\geq (0.9412+o(1))N_1(T). % 3/2 - 1/2 * ??
\]
\end{corollary}

The best previous result is due to Farmer, Gonek and Lee \cite{FGL}, showing that more than $85.83\% $ of the zeros of $\xi'(s)$ are simple. 

\section{Derivation of the optimization problems}\label{prelims}

Let $\A_{LP}$ be the class of even continuous functions $f\in L^1(\r)$ satisfying the following conditions:
\begin{enumerate}
\item $\ft f(0)=f(0)=1$;
\item $\ft f\geq 0$;
\item $f$ is eventually non-positive.
\end{enumerate}
By eventually non-positive we mean that $f(x)\leq 0$ for all sufficiently  large $|x|$. We then define the last sign change of $f$ by
\begin{equation*}
r(f)=\inf \big\{r>0: f(x)\leq 0 \text{ for } |x|\geq r\big\}.
\end{equation*}
It is easy to show that if $f\in\A_{LP}$, then $\ft f \in L^1(\r)$.

A remarkable breakthrough in the sphere problem was achieved by Cohn and Elkies in \cite{CE}, where they showed that if $\Delta(\r^d)$ is the highest sphere packing density in $\r^d$, then
$$
\Delta(\r^d) \leq \Q(f)
$$
for any $f\in  \A_{LP}(\r^d)$ (this is the analogous class in higher dimensions defined for radial functions $f$), where
$$
\Q(f) = \frac{\pi^{d/2}}{(d/2)!2^{d}}r(f)^d.
$$ 
With this approach they generated numerical upper bounds, called linear programming bounds, for the packing density for dimensions up to $36$ (nowadays it goes much higher) that improved every single upper bound known at the time and still are the current best. These upper bounds in dimensions $8$ and $24$ revealed to be extremely close to the lower bounds given by the $E_8$ root lattice and the $\Lambda_{24}$ Leech lattice, revealing that in these special dimensions the linear programming approach could exactly act as the dual problem. This is what inspired Viazovska et. al. \cite{V,CKMRV} to follow their program and solve the sphere packing problem in dimensions $8$ and $24$. What is interesting and surprising to us is that the same space $\A_{LP}$ can be used (but with a functional different than $\Q(f)$) to produce numerical bounds in analytic number theory. 

The general strategy to study problems \eqref{prob1} and \eqref{prob2} is based on Montgomery's function
\begin{align*}%\label{pair-correlation-function}
F(x,T)=\frac{1}{N(T)}\sum_{0<\gamma, \gamma'\leq T}T^{ix(\gamma-\gamma')}w(\gamma-\gamma'),
\end{align*}
where the sum is over pairs of ordinates of zeros (with multiplicity) of $\zeta(s)$ and $w(u)=\frac{4}{4+u^2}$. For each $T$, the function $x \mapsto F(x,T)$ is even, real, and as observed independently by Mueller and Heath-Brown, non-negative. The first step is to use Fourier inversion to obtain
\begin{equation}\label{explicit_formula}
\displaystyle\sum_{0<\gamma, \gamma'\leq T}g\bigg((\gamma-\gamma')\dfrac{\log T}{2\pi}\bigg)w(\gamma-\gamma')=N(T)\int_{-\infty}^{\infty}\widehat{g}(x)F(x,T)\d x,
\end{equation}
for suitable functions $g$, and use some known asymptotic estimate for $F(x,T)$ as $T\to\infty$ (which is proven only under RH or GRH). Secondly, after a series of inequalities, we produce a minimization problem over $\A_{LP}$ for some functional $\Z$. We then approach the problem numerically, using the class of functions used for the sphere packing problem in \cite{CE} and sum-of-squares/semidefinite programming techniques to optimize over these functions. The same basic strategy can be, in principle, carried out for other functions where we have a pair correlation approach. Indeed, we will also derive functionals related to the zeros of $\xi'(s)$ and a certain average of primitive Dirichlet L-functions.

\subsection{Bounding $N^*(T)$ and $N(x,T)$}
Ultimately, the functionals we need to define depend on the asymptotic behavior of $F(x,T)$. To analyze the function $N^{*}(T)$ we define the functionals
\[
\Z(f)  =r(f) + \frac{2}{r(f)}\int_0^{r(f)}f(x)x\d x
\]
and
\[
\wt \Z(f)=r(f) + \frac{2}{r(f)}\int_0^{r(f)}f(x)x\d x + 3\int_{r(f)}^{\frac{3}{2}r(f)}f(x)\d x -  \frac{2}{r(f)}\int_{r(f)}^{\frac{3}{2}r(f)}f(x)x\d x.
\]

\begin{lemma}\label{thmsimplezerosofxi}
Let $f\in \A_{LP}$. Assuming RH, we have
\[
N^*(T)\leq (\Z(f)+o(1))N(T).
\]
Assuming GRH, for every fixed sufficiently small $\delta>0$, we have
\[
N^*(T)\leq (\wt \Z(f) + O(\delta) + o(1))N(T).
\]
\end{lemma}

\begin{proof}
We start assuming only RH. Refining the original work of Montgomery \cite{M}, Goldston and Montgomery \cite[Lemma 8]{GM} proved that
\begin{equation}\label{Fasymptotics}
F(x,T) = \big(T^{-2|x|}\log T+|x|\big)(1+o(1)),
\end{equation}
uniformly for $|x|\leq 1$. Let $f\in \A_{LP}$ and let $g(x)=\ft f(x/r(f))/r(f)$. We can then use formula \eqref{explicit_formula} in conjunction with the asymptotic formula above to obtain
\begin{align*}
\sum_{0<\gamma, \gamma'\leq T}  g\bigg((\gamma-\gamma')\dfrac{\log T}{2\pi}\bigg)w(\gamma-\gamma')  = N(T)\left[\widehat{g}(0) %\int_{-1}^{1}{\ft g}(x)T^{-2|x|}\log T\d x 
+\int_{-1}^{1}{\ft g}(x)|x|\d x + \int_{|x|>1}{\ft g}(x)F(x,T)\d x + o(1)
\right],
\end{align*}
where the $o(1)$ above is justified since $\ft g$ is continuous and $T^{-2|x|}\log T\to{\bo  \delta}_0(x)$ as $T\to\infty$ (in the distributional sense). Moreover, since $F(x,T)$ is non-negative and $\ft g(x)\leq 0$ for $|x|\geq 1$ we deduce that 
\begin{align*}
 \sum_{0<\gamma, \gamma'\leq T}g\bigg((\gamma-\gamma')\dfrac{\log T}{2\pi}\bigg)w(\gamma-\gamma') &  \leq N(T)\left[ \ft g(0)+2\int_{0}^{1}{\ft g}(x)x\d x + o(1)\right]  = N(T)\left[\frac{\Z(f)}{r(f)} + o(1)\right].
\end{align*}
On the other hand, clearly we have
\begin{equation}\label{eq:oneineq}
\sum_{0<\gamma, \gamma'\leq T} g\bigg((\gamma-\gamma')\dfrac{\log T}{2\pi}\bigg)w(\gamma-\gamma') \geq g(0)\sum_{0<\gamma\leq T} m_\rho = \frac{N^*(T)}{r(f)}.
\end{equation}
Combining these results we show the first inequality in the theorem.  

Assume now GRH. It is then shown in \cite{GGOS} that for any fixed and sufficiently small $\delta >0$ we have
\begin{equation}\label{Fasymptoticmore}
F(x,T) \geq \frac{3}{2}-|x| - o(1),
\end{equation}
uniformly for $1\leq |x|\leq \frac{3}{2}-\delta$ as $T\to \infty$. Using this estimate together with \eqref{Fasymptotics} and the fact that $\ft g(x)\leq 0$ for $|x|\geq 3/2-\delta$, we obtain
\begin{align*}
\sum_{ 0<\gamma, \gamma'\leq T}g\bigg((\gamma-\gamma')\dfrac{\log T}{2\pi}\bigg)w(\gamma-\gamma')& \leq N(T)\left[{\ft g}(0) +2\int_{0}^{1}{\ft g}(x)x\d x + 2\int_{1}^{\frac{3}{2}-\delta}{\ft g}(x)\left(\frac{3}{2}-x\right)\d x + o(1)
\right] \\
& = N(T)\left[\frac{\wt \Z(f)}{r(f)}+o(1) +O(\delta)\right].
\end{align*}
Arguing as before, using \eqref{eq:oneineq}, we complete the proof of the lemma.
\end{proof}

To analyze $N(x,T)$ we define the functional
\[
\P(f) = \inf \big\{\la>0: p_f(\la) > 0 \big\},
\]
where 
\[
p_f(\la)=-1+\frac{\la}{r(f)}+ \frac{2r(f)}{\la}\int_0^{\frac{\la}{r(f)}}\ft f(x)x\d x,
\]
and the functional
\[
\wt \P(f) = \inf\big\{\la>0: \wt p_f(\la) > 0 \big\},
\]
where 
$$
\wt p_f(\la)=-1+\frac{\la}{r(f)}+ \frac{2r(f)}{\la}\int_0^{\frac{\la}{r(f)}}\ft f(x)x\d x +3\int_{\frac{\la}{r(f)}}^{\frac{3\la}{2r(f)}}\ft f(x)\d x -  \frac{2r(f)}{\la}\int_{\frac{\la}{r(f)}}^{\frac{3\la}{2r(f)}}\ft f(x)x\d x.
$$
Note that these functionals are well defined since $p_f$ and $\wt p_f$ are $C^1$ functions that assume $-1$ at $\la=0$, and using the fact that $\widehat{f}\in L^1(\mathbb{R})$ one can show 
$$
\lim_{\la\to\infty} \frac{p_f(\la)}{\la} = \lim_{\la\to\infty} \frac{\wt p_f(\la)}{\la} =\frac{1}{r(f)} > 0.
$$
\begin{lemma}\label{thmsmallgapszerosofxi}
Let $f\in \A_{LP}$ and  $\ep>0$. Assuming RH and \eqref{NstarNconjecure}, we have
$$
N(\P(f)+\varepsilon,T)\gg N(T).
$$
Assuming GRH and \eqref{NstarNconjecure}, we have
$$
N(\wt \P(f)+\varepsilon,T)\gg N(T).
$$
\end{lemma}
\begin{proof}
Let $f\in\A_{LP}$, $\lambda>0$, and set $g(x)= f(r(f)x/\la)$. Assuming RH, we have 
\begin{align*}
\sum_{0<\gamma, \gamma'\leq T}g\bigg((\gamma-\gamma')\dfrac{\log T}{2\pi}\bigg)w(\gamma-\gamma') 
&= N(T)\int_{-\infty}^{\infty}\widehat{g}(x)F(x,T)\d x\\
&\geq N(T)\left[\ft g(0)+2\int_{0}^{1}\ft g(x)x\d x+o(1)\right]
 \\ & = N(T)\left[1+p_f(\la)+o(1)\right].
\end{align*}
Applying formula \eqref{explicit_formula} in conjunction with \eqref{Fasymptotics}, while assuming GRH, and using in addition \eqref{Fasymptoticmore}, we have
\begin{align*}
\sum_{0<\gamma, \gamma'\leq T}g\bigg((\gamma-\gamma')\dfrac{\log T}{2\pi}\bigg)w(\gamma-\gamma') 
&= N(T)\int_{-\infty}^{\infty}\widehat{g}(x)F(x,T)\d x\\
&\geq N(T)\left[\ft g(0)+2\int_{0}^{1}\ft g(x)x\d x + 2\int_1^{\frac32-\delta}\ft g(x)\left(\frac{3}{2}-x\right)\d x + o(1)\right]\\
&= N(T)\left[1+\wt p_f(\la)+o(1)\right].
\end{align*}
Since $\ft f\geq 0$, we have $\|f\|_{\infty}=f(0)=1$. Recall now the pair correlation function $N(x,T)$ defined in \eqref{paircorrelfunction}. We have
\begin{align*}
\sum_{0<\gamma, \gamma'\leq T}g\bigg((\gamma-\gamma')\dfrac{\log T}{2\pi}\bigg)w(\gamma-\gamma') & = N^*(T) +2\sum_{\substack{0<\gamma, \gamma'\leq T\\ 0<\gamma-\gamma'}}f\bigg((\gamma-\gamma')\dfrac{r(f)\log T}{2\pi\la}\bigg)w(\gamma-\gamma')  \\
& \leq N^*(T) +2\sum_{\substack{0<\gamma, \gamma'\leq T\\ 0<\gamma-\gamma' \leq \frac{2\pi \la}{\log T}}}f\bigg((\gamma-\gamma')\dfrac{r(f)\log T}{2\pi\la}\bigg)w(\gamma-\gamma') \\
& \leq N^*(T)+2N(\la,T) \\ & = (1+o(1))N(T) + 2N(\la,T),
\end{align*}
where in the last step we have used \eqref{NstarNconjecure}. We then obtain, assuming RH, that 
\begin{align*} %\label{5_9_19:21pm}
\frac{N(\la,T)}{N(T)}\geq \frac{p_f(\la)}{2}+o(1).
\end{align*}
Similarly assuming GRH. Noting that $N(\la,T)$ increases with $\la$, so we can choose $\la$ arbitrarily close to $\P(f)$, we obtain the desired result.
\end{proof}

\subsection{Bounding $N_{\Phi,s}(Q)$}
Define the following functional over $\A_{LP}$
$$
\L(f) = \frac{r(f)}{2} + \dfrac{4}{r(f)}\int_{0}^{\frac{r(f)}{2}}f(x)x\d x+2\int_{\frac{r(f)}{2}}^{r(f)}f( x)\d x.
$$
We have the following lemma.
\begin{lemma}\label{thmnphi}
Let $f\in\A_{LP}$. Assuming GRH, for every fixed small $\delta>0$ we have
$$
N_{\Phi,s}(Q) \geq (2- \L(f) + O(\delta) + o(1))N_{\Phi}(Q).
$$
\end{lemma}
\begin{proof}
For $Q>1$ and $x\in\r$, we define the pair correlation function $F_\Phi$ by
\begin{align*}
F_\Phi(Q^{x},W)=\dfrac{1}{N_{\Phi}(Q)}\displaystyle\sum_{Q\leq q\leq 2Q}\dfrac{W(q/Q)}{\varphi(q)}\displaystyle\sum_{\substack{\chi \ (\mathrm{mod} \ q)\\ \mathrm{primitive}}}\bigg|\displaystyle\sum_{\gamma_\chi}\mathcal{M}\Phi(i\gamma_\chi)Q^{i\gamma_\chi{x}}\bigg|^2.
\end{align*}
Using the asymptotic large sieve, Chandee, Lee, Liu and Radziwi\l\l \,\,\cite{CLLR} showed the following asymptotic formula under GRH
\begin{align}\label{FPhiasymp}
& F_{\Phi}(Q^{{x}},W)  \\ & = (1+o(1))\bigg[1-(1-|x|)_++\Phi\big(Q^{-|{x}|}\big)^2\log Q \bigg(\dfrac{1}{2\pi}\int_{-\infty}^{\infty}\big|\mathcal{M}\Phi(it)\big|^2\d t\bigg)^{-1}\bigg] + O\Big(\Phi(Q^{-|{x}|})\log^{1/2} Q\Big)\nonumber,
\end{align}
which holds uniformly for $|{x}|\leq 2-\delta$ as $Q\to\infty$, for any fixed and sufficiently small $\delta>0$. Let
$$
N^*_{\Phi}(Q):=\displaystyle\sum_{Q\leq q\leq 2Q}\dfrac{W(q/Q)}{\varphi(q)}\displaystyle\sum_{\substack{\chi \ (\mathrm{mod} \ q) \\ \text{primitive}}}\displaystyle\sum_{\gamma_\chi}m_{\rho_\chi}\big|\mathcal{M}\Phi(i\gamma_\chi)\big|^2,
$$
where $m_{\rho_\chi}$ denote the multiplicity of the nontrivial zero $\rho_\chi=\hh+i\gamma_\chi$ of $L(s,\chi)$. Since
$$
\displaystyle\sum_{\substack{\gamma_\chi\\\text{simple}}}\big|\mathcal{M}\Phi(i\gamma_\chi)\big|^2\geq \displaystyle\sum_{\gamma_\chi}(2-m_{\rho_\chi})\big|\mathcal{M}\Phi(i\gamma_\chi)\big|^2
$$
we obtain
\begin{align} \label{ineqNstarwithN}
N_{\Phi,s}(Q)\geq 2N_{\Phi}(Q)-N^{*}_{\Phi}(Q).
\end{align}
For any $g\in L^1(\r)$ with $\ft g\in L^1(\r)$ we have the following formula (Fourier inversion):
\begin{align*}
\sum_{Q\leq q\leq 2Q} \dfrac{W(q/Q)}{\varphi(q)}\displaystyle\sum_{\substack{\chi \ (\mathrm{mod} \ q) \\ \text{primitive}}}\displaystyle\;\sum_{\gamma_\chi, \gamma'_\chi}\mathcal{M}\Phi(i\gamma_\chi)\mathcal{M}\Phi(i\gamma'_\chi)\,\widehat{g}\bigg(\dfrac{(\gamma_\chi-\gamma'_\chi)\log Q}{2\pi}\bigg)
	 =N_{\Phi}(Q)\int_{-\infty}^{\infty}g(x)F_{\Phi}(Q^x,W)\d x .
\end{align*}
Letting $f\in\A_{LP}$ and $g(x)=f(r(f)x/(2-\delta))$, for any primitive character $\chi  \ (\mathrm{mod} \ q)$ we obtain
\begin{align*}
&\sum_{\gamma_\chi, \gamma'_\chi}\mathcal{M}\Phi(i\gamma_\chi)\mathcal{M}\Phi(i\gamma'_\chi)\,\widehat{g}\bigg(\dfrac{(\gamma_\chi-\gamma'_\chi)\log Q}{2\pi}\bigg) \\ & = \displaystyle\sum_{\gamma_\chi}m_{\rho_\chi}\big|\mathcal{M}\Phi(i\gamma_\chi)\big|^2\,\widehat{g}(0) \nonumber + \displaystyle\sum_{\gamma_\chi\neq \gamma'_\chi}\mathcal{M}\Phi(i\gamma_\chi)\mathcal{M}\Phi(i\gamma'_\chi)\,\widehat{g}\bigg(\dfrac{(\gamma_\chi-\gamma'_\chi)\log Q}{2\pi}\bigg) \nonumber \\
    	& \geq \dfrac{2-\delta}{r(f)}\,\displaystyle\sum_{\gamma_\chi}m_{\rho_\chi}\big|\mathcal{M}\Phi(i\gamma_\chi)\big|^2. 
    	\end{align*}
This implies that
\begin{align*}
& \sum_{Q\leq q\leq 2Q} \dfrac{W(q/Q)}{\varphi(q)}\displaystyle\sum_{\substack{\chi \ (\mathrm{mod} \ q) \\ \text{primitive}}}\displaystyle\;\sum_{\gamma_\chi, \gamma'_\chi}\mathcal{M}\Phi(i\gamma_\chi)\mathcal{M}\Phi(i\gamma'_\chi)g\bigg(\dfrac{(\gamma_\chi-\gamma'_\chi)\log Q}{2\pi}\bigg)  \geq \dfrac{2-\delta}{r(f)}N^*_{\Phi}(Q).
\end{align*}
On the other hand, observing that 
$$
\Phi\big(Q^{-|{x}|}\big)^2\log Q \bigg(\dfrac{1}{2\pi}\int_{-\infty}^{\infty}\big|\mathcal{M}\Phi(it)\big|^2\d t\Big)\to {\bo \delta}(x),
$$
as $Q\to \infty$ (in the distributional sense) and that
$$
\log^{1/2} Q \int_{-(2-\delta)}^{2-\delta} g(x) \Phi(Q^{-|x|})\d x \leq 2\log^{-1/2} Q \int_{Q^{-({2-\delta})}}^1 \Phi(t)\frac{\d t}{t} = O(\log^{-1/2} Q),
$$
we can use the asymptotic estimate \eqref{FPhiasymp} to obtain
\begin{align*}
\int_{-\infty}^{\infty}g({x})F_{\Phi}(Q^{x},W)\d{x}
  & \leq \int_{-(2-\delta)}^{2-\delta}g({x})F_{\Phi}(Q^{x},W)\d{x} \\ & =  g(0)+ \int_{-(2-\delta)}^{2-\delta}g({x})(1-(1-|x|)_+)\d x + O(\log^{-1/2} Q) + o(1)
\\ & =\frac{2 \L(f)}{r(f)} + O(\delta) + o(1).
\end{align*}
We then conclude that 
$$N^*_{\Phi}(Q) \leq N_{\Phi}(Q) \left(\L(f) + O(\delta) + o(1) \right).$$
Using \eqref{ineqNstarwithN} we finish the proof.
\end{proof}

\subsection{Bounding $N_1^*(T)$}
Similarly to the case of the Riemann zeta-function, the functionals that we need to define depend on the asymptotic behavior of the function $F_1(x,T)$ defined by 
\begin{align} \label{5_9_18:57pm}
F_1(x,T)=N_1(T)^{-1}\displaystyle\sum_{0<\gamma_1, \gamma'_1 \leq T}T^{ix(\gamma_1-\gamma'_1)}w(\gamma_1-\gamma'_1),
\end{align}
where $x\in\mathbb{R}$, $T>0$ and the sum is over pairs of ordinates of zeros (with multiplicity) of $\xi'(s)$. To analyze $N_1^{*}(T)$ we define the following functional
\begin{align*}
\Z_{1}(f) =r(f)+\dfrac{2}{r(f)}\int_{0}^{r(f)}x\, f(x)\d x-\dfrac{8}{r(f)^2}\int_{0}^{r(f)}x^2\, f(x)\d x  +\displaystyle\sum_{k=1}^{\infty}\dfrac{2c_k}{r(f)^{2k+1}}\int_{0}^{r(f)}x^{2k+1}\, f(x)\d x,
\end{align*}
where $c_k=2^{2k+1}\frac{(k-1)!}{(2k)!}$. 

\begin{lemma}\label{thmsimplezerosofxiprime}
Let $f\in \A_{LP}$. Assuming RH, for every fixed small $\delta >0$ we have
\[
N_1^*(T)\leq (\Z_1(f) + O(\delta) + o(1))N_1(T).
\]
\end{lemma}
\begin{proof}
A result similar to \eqref{Fasymptotics} for the function $F_1(x,T)$ defined in \eqref{5_9_18:57pm} is also known (see \cite[Theorem 1.1]{FGL}), which is the following: for any fixed small $\delta>0$ we have
$$
F_1(x,T)=T^{-2|x|}\log T+|x| - 4|x|^2+\displaystyle\sum_{k=1}^{\infty}c_k |x|^{2k+1}+ o(1)(1+T^{-2|x|}\log T),
	$$ 
uniformly for $|x|\leq 1-\delta$ as $T\to\infty$, where $c_k=2^{2k+1}\frac{(k-1)!}{(2k)!}$.  The proof then follows the same strategy as the proof for $\zeta(s)$ and we leave the details to the reader.
\end{proof}
 	
\section{Numerically optimizing the bounds}%\label{numerics}

Going back to the sphere packing problem, since we obviously have $\Delta(\r^1)=1$, this shows $r(f)\geq 1$ for all $f\in \A_{LP}$. The last sign change equals $1$ for two (suspiciously) well-known functions: the hat function
\begin{equation*}
H(x)=(1-|x|)_+,
\end{equation*}
whose Fourier transform is $\ft H(x)=\frac{\sin^2(\pi x)}{(\pi x)^2}$, and Selberg's function
\begin{equation*}
S(x)=\frac{\sin^2(\pi x)}{(\pi x)^2(1-x^2)},
\end{equation*}
whose Fourier transform is supported in $[-1,1]$ and given by $\ft S(x)=1-|x|+\frac{\sin(2\pi x)}{2\pi}$ for $|x|<1$. In particular, we can use these two functions to evaluate the functionals derived in Section \ref{prelims} to obtain bounds, but this does not result in the best possible bounds. To obtain better bounds we use the class of functions used in the linear programming bounds by Cohn and Elkies \cite{CE} for sphere packing. That is, we consider the subspace $\A_{LP}(d)$ consisting of the functions 
$f \in \A_{LP}$ of the form 
\begin{equation}\label{eq:subspace} 
f(x) = p(x) e^{-\pi x^2},
\end{equation}
where $p$ is an even polynomial of degree $2d$. 

In \cite{CE}, optimization over a closely related class of functions is done by specifying the functions by their real roots and optimizing the root locations. For the sphere packing problem this works very well, where in $\r^{24}$ it leads to a density upper bound that is sharp to within a factor $1 + 10^{-51}$ of the optimal configuration \cite{CM}. We have also tried this approach for the optimization problems in this paper, but this did not work very well because the optimal functions seem to have very few real roots, which produces a strange effect in the numerical computations, where the last forced root tends to diverge when you increase the degree of the polynomial\footnote{It is worth mentioning that, in a related uncertainty problem, Cohn and Gon\c{c}alves \cite{CGo} discovered the same kind of instability in low dimensions.}. Instead we use sum-of-squares characterizations and semidefinite programming, as was done in \cite{dLOV} for the binary sphere packing problem.

\emph{Semidefinite programming} is the optimization of a linear functional over the intersection of a cone of positive semidefinite matrices (real symmetric matrices with nonnegative eigenvalues) and an affine space. A semidefinite program is often given in block form, which can be written as
\begin{align*}
\text{minimize} \; \sum_{i=1}^I \mathrm{tr}(X_iC_i) : \; & \sum_{i=1}^I \mathrm{tr}(X_iA_{i,j}) = b_j \text{ for } j \in [m],\\
& \,X_1,\ldots,X_I \in \mathbb R^{n \times n} \text{ positive semidefinite},
\end{align*}
where $I\in \n$ gives the number of blocks, $\{C_i\} \subseteq \r^{n \times n}$ is the objective, and $\{A_{i,j}\} \subseteq \r^{n \times n}$, $b \in \r^m$ give the linear constraints (for notational simplicity we take all blocks to have the same size).
Semidefinite programming is a broad generalization of linear programming (which we recover by setting $n=1$ in the above formulation), and, as for linear programming, there exist efficient algorithms for solving them. The reason semidefinite programming comes into play here, is that we can model polynomial inequality constraints as sum-of-squares constraints, which in turn can be written as semidefinite constraints; see, e.g., \cite{Bl}. 

\subsection{Proof of Theorems \ref{thmNstar}, \ref{thmsimplezerosofdir}, and \ref{thmn1star}}

To obtain the first part of Theorem~\ref{thmNstar} from Lemma~\ref{thmsimplezerosofxi} we need to minimize the functional $\Z$ over the space $\A_{LP}(d)$. We can see this as a bilevel optimization problem, where we optimize over scalars $R \geq 1$ in the outer problem, and over functions $f \in \A_{LP}(d)$ satisfying $r(f) = R$ in the inner problem. The outer problem is a simple one dimensional optimization problem for which we use Brent's method \cite{Br}.

A polynomial $p$ that is nonnegative on $[R,\infty)$ can be written as $s_1(x) + (x-R) s_2(x)$, where $s_1$ and $s_2$ are sum-of-squares polynomials with $\mathrm{deg}(s_1),\mathrm{deg}(s_2(x))+1 \leq \deg(p)$; see, e.g., \cite{PS}. This shows that functions of the form \eqref{eq:subspace} that are non-positive on $[R,\infty)$ can be written as
\[
f(x) = -\big(s_1(x^2) + (x^2-R^2) s_2(x^2)\big) e^{-\pi x^2}.
\]

Let $v(x)$ be a vector whose entries form a basis of the univariate polynomials of degree at most $d$. The polynomials $s_1$ and $s_2$ are sum-of-squares if and only if they can be written as $s_i(x) = v(x)^{\sf T} X_i v(x)$ for some positive semidefinite matrices $X_i$ of size $d+1$. That is, we can parameterize functions of the form \eqref{eq:subspace} that are non-positive on $[R,\infty)$ by two positive semidefinite  matrices $X_1$ and $X_2$ of size $d+1$. 

The space of functions of the form \eqref{eq:subspace} is invariant under the Fourier transform. Since a polynomial of degree $2d$ that is nonnegative on $[0,\infty)$ can be written as $s_3(x) + x s_4(x)$, where $s_i(x) = v(x)^{\sf T} X_i v(x)$ for $i=3,4$ are sum-of-squares polynomials of degree $2d$, we have that $\widehat f$ is of the form
\[
\ft{f}(x) = \big(s_3(x^2) + x^2 s_4(x^2)\big)  e^{-\pi x^2}.
\]

Let $\mathcal T$ be the operator that maps $x^{2k}$ to the function $\frac{k!}{\pi^k} L_k^{-1/2}(\pi x^2)$,
where $L_k$ is the Laguerre polynomial of degree $k$ with parameter $-1/2$. Then, for $p$ an even polynomial, we have that $(\mathcal Tp)(x) e^{-\pi x^2}$ is the Fourier transform of $p(x) e^{-\pi x^2}$. We can now describe the functions of the form \eqref{eq:subspace} that are non-positive on $[R,\infty)$ and have nonnegative Fourier transform by positive semidefinite matrices $X_1,\ldots,X_4$ of size $d+1$ whose entries satisfy the linear relations coming from the identity $I(X_1,\ldots,X_4) = 0$, where
\begin{equation*}
I(X_1,\ldots,X_4) = \mathcal T\big(-s_1(x^2) - (x^2-R^2) s_2(x^2)\big) - \big(s_3(x^2) + x^2 s_4(x^2)\big).
\end{equation*}
Here $\mathcal T(-s_1(x^2) - (x^2-R^2) s_2(x^2))$ is a polynomial whose coefficients are linear combinations in the entries of $X_1$ and $X_2$, and the same for $s_3(x^2) + x^2 s_4(x^2)$ with $X_3$ and $X_4$. The linear constraints on the entries of $X_1,\ldots,X_4$ are then obtained by expressing $I(X_1,\ldots,X_4)$ in some polynomial basis and setting the coefficients to zero.

The conditions $f(0) = 1$ and $f(R) = 0$ are linear in the entries of $X_1$ and $X_2$, and the condition $\ft f(0)=1$ is a linear condition on the entries of $X_3$ and $X_4$.
Finally, the objective $\Z(f)$ is a linear combination in the entries of $X_1$ and $X_2$, which can be implemented by using the identity
\begin{equation}\label{eq:iidentity}
\int x^m e^{-\pi x^2} \d x = -\frac{1}{2\pi^{m/2+1/2}} \Gamma\Big(\frac{m+1}{2}, \pi x^2\Big),
\end{equation}
where $\Gamma$ is the upper incomplete gamma function. Hence, the problem of minimizing $\Z(f)$ over functions $f \in \A_{LP}(d)$ that satisfy $r(f) = R$ is a semidefinite program.

To obtain the second part of Theorem~\ref{thmNstar} from Lemma~\ref{thmsimplezerosofxi} and to obtain Theorem~\ref{thmsimplezerosofdir} from Lemma~\ref{thmnphi} we use the same approach with a different functional. To obtain Theorem~\ref{thmn1star} from Lemma~\ref{thmsimplezerosofxiprime} we also do the same as above, but now truncate the series in the functional $\Z_{1}$ at $k=15$ and add the easy to compute upper bound $10^{-10}$ on the remainder of the terms.

\subsubsection{Implementation and numerical issues} \label{sec:rig}  In implementing the above as a semidefinite program we have to make two choices for the polynomial basis that we use: the basis defining the vector $v(x)$, and the basis to enforce the identity $I(X_1,\ldots,X_4) = 0$. This choice of bases is important for the numerical conditioning of the resulting semidefinite program. Following \cite{dLOV} we choose the Laguerre basis $\{L_n^{-1/2}(2\pi x^2)\}$, as this seems natural and performs well in practice (it multiplied by $e^{-\pi x^2}$ is the complete set of even eigenfunctions of the Fourier transform). We solve the semidefinite programs using sdpa-gmp \cite{Nakata}, which is a primal-dual interior point solver using high precision floating point arithmetic. For the code to generate the semidefinite programs and to perform the post processing we use Julia \cite{BEKV}, Nemo \cite{FHHJ}, and Arb \cite{J} (where we use Arb for the ball arithmetic used in the verification procedure). For all computations we use $d=40$. In solving the systems we observe that $X_1$ can be set to zero everywhere without affecting the bounds, so that $r(f) = R$ holds exactly for the function $f(x) = (R^2 - x^2) v(x^2)^{\sf T} X_2 v(x^2) e^{-\pi x^2}$ defined by $X_2$. 

The above optimization approach uses floating point arithmetic and a numerical interior point solver. This means the identity $I(0, X_2, X_3, X_4) = 0$ will not be satisfied exactly, and, moreover, because the solver can take infeasible steps the matrices $X_2$, $X_3$, and $X_4$ typically have some eigenvalues that are slightly negative. In practice this leads to incorrect upper bounds if the floating point precision is not high enough in relation to the degree $d$. Here we explain the procedure we use to obtain bounds that are guaranteed to be correct. This is an adaptation of the method from  \cite{Lo} and \cite{dLOV}. 

We first solve the above optimization problem numerically to find $R$ and $f$ for which we have a good objective value $v = \L(f)$. Then we solve the semidefinite program again for the same value of $R$, but now we solve it as a feasibility problem with the additional constraint $\L(f) \leq v + 10^{-6}$. The interior point solver will try to give the analytical center of the semidefinite program, so that typically the matrices are all positive definite; that is, the eigenvalues are all strictly positive.  Then we use interval arithmetic to check rigorously that $X_2$, $X_3$, and $X_4$ are positive definite, and we compute a rigorous lower bound $b$ on the smallest eigenvalues of $X_3$ and $X_4$.

Using interval arithmetic we compute an upper bound $B$ on the largest coefficient of $I(0, X_2, X_3,X_4)$ in the basis given by the $2d+1$ entries on the diagonal and upper diagonal of the matrix $(R^2 - x^2) v(x^2) v(x^2)^{\sf T}$. If $b \geq (1+2d)B$, then it follows that it is possible to modify the corresponding entries in $X_3$ and $X_4$ such that these matrices stay positive definite and such that $I(0,X_2,X_3,X_4) = 0$ holds exactly \cite{Lo}. This proves that the Fourier transform of the function $f(x) = (R^2 - x^2) v(x^2)^{\sf T} X_2 v(x^2)e^{-\pi x^2}$ is nonnegative.

The only remaining problem is that the identities $f(0) = 1$ and $\mathcal T f(0) = 1$ will not hold exactly. We can, however, for instance write the first part of Theorem~\ref{thmNstar} as follows: Suppose $f$ is a continuous function in $L^1(\r)$ with $f(x) \leq 0$ for $|x| \geq R$ and with $\widehat f \geq 0$, then 
$
N^*(T) \leq (\Z(f) + o(1)) N(T),
$
where we use the following modified definition for $\Z(f)$:
\[
\Z(f) = \frac{1}{\widehat f(0)} \left(f(0)r(f) + \frac{2}{r(f)}\int_0^{r(f)}f(x)x\d x\right).
\]
Since the function $f$ defined by $X_1$ has been verified to satisfy all the constraints, the only thing we still need to do is to compute a rigorous upper bound on $\Z(f)$ (or on similar modifications of the functionals $\tilde \Z(f)$, $\Z_1(f)$, or $\L(f)$), for which we use identity \eqref{eq:iidentity} and interval arithmetic.

\begin{remark} In the arXiv version of this paper we attach the files `Z-$40$.txt', `tildeZ-$40$.txt', `L-$40$.txt', and `Z1-$40$.txt' that contain the value of $R$ on the first line and the matrices $X_2, X_3$ and $X_4$ on the next $3$ lines (all in $100$ decimal floating point values). For convenience it also contains the coefficients of $f$ in the monomial basis on the last line (but these are not used in the verification procedure). We include a script to perform the above verification and compute the bounds rigorously, as well as the code for setting up the semidefinite programs, using a custom semidefinite programming specification library.
\end{remark}

\subsection{Proof of Theorem~\ref{thmpaircorreltation}}

To obtain the first part of Theorem~\ref{thmpaircorreltation} from Lemma~\ref{thmsmallgapszerosofxi} we need to minimize the function $\P$ over the space $\A_{LP}$. We can formulate this as a bilevel optimization problem in which we optimize over $R \geq 1$ in the outer problem. In the inner problem we perform a binary search over $\Lambda$ to find the smallest $\Lambda$ for which there exists a function $f \in \A_{LP}(d)$ that satisfies $f(R) = 0$, $f(x) \leq 0$ for $|x| \geq R$, and $p_f(\Lambda) \geq 0$.

To get a bound whose correctness we can verify rigorously we replace the constraints $f(0)=1$, $\ft f(0) = 1$, and $p_f(\Lambda) \geq 0$ by $f(0) = 1-10^{-10}$, $\ft f(0) = 1+10^{-10}$, and $p_f(\Lambda) \geq 10^{-10}$. We then use the above optimization approach to find good values for $R$ and $\Lambda$. We then add $10^{-6}$ to $\Lambda$ and solve the feasibility problem again to get the strictly feasible matrices $X_2,X_3$, and $X_4$.  By performing the same procedure as in \ref{sec:rig} we can verify that the Fourier transform of the function $f$ defined by $X_2$ is nonnegative everywhere, and using interval arithmetic we can check that the inequalities $f(0) \leq 1$, $\ft f(0) \geq 1$, and $p_f(\Lambda) > 0$ all hold. Note that this verification procedure does not actually check that $\Lambda$ is equal to or even close to $\P(f)$, but the proof of Lemma~\ref{thmsmallgapszerosofxi} also works if we replace $\P(f)$  by any $\Lambda$ for which $p_f(\Lambda)$ is strictly positive. To obtain the second part of the theorem, we do the same except that we replace $p_f$ by $\wt p_f$.

\begin{remark}
In the arXiv version of this paper we attach the files `P-$40$.txt', `tildeP-$40$.txt', that have the same layout as the files mentioned in \ref{sec:rig}, with an additional line containing the value of $\Lambda$. We again include the code to perform the verification and to produce the files.
\end{remark}

\section*{Acknowledgments}
We are very thankful to Emanuel Carneiro and Micah Milinovich for the helpful discussions, suggestions and references. We also thank the anonymous referee for valuable suggestions.

\end{document}